\documentclass{amsart}

\usepackage{amscd,amssymb}

\usepackage{graphicx}
\usepackage{hyperref}

\usepackage{kpfonts}  
\usepackage[T1]{fontenc}
\usepackage[cp1250]{inputenc}
\usepackage[polish,english]{babel}

\newtheorem{lem}{Lemma}[section]

\newtheorem{thm}[lem]{Theorem}
\newtheorem{prop}[lem]{Proposition}

\newtheorem{fact}[lem]{Fact}

\makeatletter
\def\th@plain{%
  \thm@notefont{}
  \itshape 
}
\def\th@definition{%
  \thm@notefont{}
  \normalfont 
}
\def\th@procedure{
	\thm@notefont{}
	\ttfamily
}
\makeatother

\theoremstyle{definition} 

\newtheorem{rk}[lem]{Remark}

\theoremstyle{procedure} 

\newcommand{\Za}{\mathbb{Z}}
\newcommand{\Na}{\mathbb{N}}
\newcommand{\re}{\mathbb{R}}

\newcommand{\tr}{\triangle}
\newcommand{\supp}{{\rm supp}}

\author{Maria Michalska}
\author{Justyna Walewska}
\title[Milnor numbers of deformations]{Milnor numbers of deformations of semi-quasi-homogeneous plane curve singularities II}
\address{Wydzia\l{} Matematyki i Informatyki, Uniwersytet \L{}\'o{}dzki, Banacha 22, 90-238 \L{}\'o{}d\'z{}, Poland}
\email{Maria.Michalska@math.uni.lodz.pl}
\email{walewska@math.uni.lodz.pl}

\begin{document}
\maketitle

\begin{abstract}
We show the possible Milnor numbers of deformations of semi-quasi-homogeneous isolated plane curve singularities. In Theorem~\ref{thm_Main} we list integers can be attained as Milnor numbers of a given semi-quasi-homogeneous singularity.
\end{abstract}

\section{Introduction}

Our main goal is to identify all possible Milnor numbers attained by deformations of plane curve singularities. First note that since~\cite{GuseinZade} it is known that not every integer less than the Milnor number of an isolated singularity $f$ has to be a Milnor number of a deformation of $f$. The sequence (or possible sequences of specializations of multiparameter deformations) of Milnor numbers attained by deformations gives interesting topological data for plane curves via adjacency of $\mu$-constant strata. Our interest in the subject stems from papers \cite{Pl14}, \cite{Bo} or~\cite{GreuelShustinL}, as well as classic~\cite{ArnoldProblems}.

The approach presented here is a continuation and fuller use of methods of~\cite{MMJW} hence this paper expands the results of~\cite{MMJW} and omits the assumption of irreducibility. Throughout this paper we will consider semi-quasi-homogeneous singularities (SQH for short) i.e. isolated singularities such that their initial term (in weighted Taylor expansion) is a weighted homogeneous isolated singularity. In particular, every semi-quasi-homogeneous singularity can be written in the form
\begin{equation}
\label{eqPostacfQSH}
f=\sum_{q\alpha + p\beta\  \geq\  pq}c_{\alpha\beta}\ x^{\alpha}y^{\beta}
\end{equation} 
for some positive integers $p,q$ such that the initial term ${\rm in}f=\sum_{q\alpha + p\beta\  =\  pq}c_{\alpha\beta}\ x^{\alpha}y^{\beta}$ is an isolated singularity. In such a case we will say that $f$ as weighs $(1/p,1/q)$ i.e. the weighs of the initial quasi-homogeneous term of $f$. 

Without loss of generality throughout this paper we assume that 
\begin{center}
$2\leq p\leq q$ and denote $q=kp+r$, $p> r\geq 0$.
\end{center}


We investigate the sequence of all Milnor numbers that can be attained via one-parameter deformations of $f$. We show that
\begin{thm}\label{thm_Main}
For semi-quasi-homogeneous singularity $f$ with weighs $(1/p, 1/q)$ 
\begin{enumerate}
\item if $p$ divides $q$, all Milnor numbers less than $\mu(f)$ are attained except for at most:

 the number $\mu(f)-(2p-1)$ for $p$ even 
 
 AND  
 
 numbers between $\mu(f)$ and $\mu (f)-(p-1)$

\item if $p$ does not divide $q$, all Milnor numbers less than $\mu(f)$ are attained except for at most:

 the numbers between $\mu(f)$ and $\mu(f)-m$ 
 
 AND 
 
 the numbers between $\mu(p,kp)$ and $\mu(p,kp)-(p-1)$ 
 
 AND 
 
  the number $\mu(p,kp)-(2p-1)$ if p even
  
AND

the number $\mu(p,q)-p$ if $p$ even and $q\equiv p-1 ({\rm mod\ }p)$
\end{enumerate}
Moreover, all these numbers are attained by linear deformations of f.
\end{thm}

This article is organised as follows. In Section~\ref{section_preliminaries} we recall some standard definitions, introduce useful notation and remark on Euclid's Algorithm. Section~\ref{section_lemmas} presents some steps toward the proof of Theorem~\ref{thm_Main}. In Section~\ref{section_Main_combinatorial} we present the combinatorial variant of the main Theorem~\ref{thm_Main} and its proof.

As a closing remark, it is important to note that in our main goal of listing integers attained as Milnor numbers of deformations we succeed only as far as non-degenerate deformations go. For degenerate deformations we are unable to present a systematic approach. Some integers missing in the list of Theorem~\ref{thm_Main} can be attained by degenerate deformations, but their occurence still seems irregular and hard to present clearly. In particular, one may check that the clever approach of Brzostowski and Krasiński~\cite{BK} does work for some singularities nicely. All suggestions or comments are welcome.

\section{Preliminaries}\label{section_preliminaries}

\subsection{ Newton numbers and diagrams}
Every singularity $f$ has a Newton diagram which here will mean the finite set of segments that give the boundary of $\supp f+\re^2_+$ (except the two half-lines). Every such chain will be called a Newton diagram without referring to a concrete singularity.

We will say that $\nu(D)$ is the Newton number of the diagram $D$ when it is equal to the Milnor number of a nondegenerate isolated singularity with the diagram~$D$ (see~\cite{Kush} for the classic equivalent combinatorial definition). Note that if both end-points of two diagrams are the same, the difference between the Newton numbers of the diagrams is equal to twice the area between the diagrams. 

For a diagram $D$ we will say that the diagram $\tilde{D}$ is its deformation if $\tilde{D}$ arises as the convex hull of $D\cup P$, where $P$ is some set of points in the non-negative quadrant of the lattice $\Za^2$. Since in such a case $\nu(D)\geq\nu(\tilde{D})$, we will also write $D\geq \tilde{D}$.

For a diagram $D$ we will say that a Newton number is attainable if it is attained by some deformation $\tilde{D}$ of~$D$.

\subsection{Some notation}
Let us introduce a convenient notation for diagrams. If $P=(p,0)$, $Q=(0,q)$ then any translation of the segment $PQ$ will be denoted as  $\tr(p,q)$, in other words
\begin{eqnarray}
\tr(p,q)& := &\text{hypotenuse of a right triangle with base}\nonumber\\
 &  &\text{of length }p\text{ and heigth }q\nonumber
\end{eqnarray}

We will write $n\tr(p,q)$ instead of $\tr(np,nq)$. Moreover, for $\tr(p_1,q_1),\dots,\tr(p_l,q_l)$ denote by 
$$(-1)^s\left(\ \tr(p_1,q_1)+\dots+\tr(p_l,q_l)\ \right)$$
any translation of a polygonal chain with endpoints $Q, Q+(-1)^s[p_1,-q_1], \dots, Q+(-1)^s\left[\sum_{i=1}^l p_i\ ,\ -\sum_{i=1}^l q_i\right].$


\subsection{Relations between singularities and the Newton diagrams}

Consider a semi-quasi-homogeneous singularity $f$ of the form~\eqref{eqPostacfQSH}. In particular, SQH singularity is nondegenerate. Hence, by the classic result~\cite{Kush}, its Milnor number is equal the Newton number of its diagram. Moreover, note that the weights $(1/p,1/q)$ of a SQH singularity $f$ such that the Newton diagram of $f$ is contained in $\tr(p,q)$ with both end-points on the axes are unique.

\begin{prop}
The only non-convenient SQH singularities are in the case $q=kp$ i.e. if $q\nequiv  0({\rm mod}\ p)$ then the diagram of $f$ is of the form $\tr(p,q)$ with end-points on both axes.
\end{prop}
\begin{proof}
If $p \leq q$, then the monomial $y$ divides $f$ if and only if the initial term of $f$ is homogeneous (i.e. weights are $(1,1)$). Whereas $x$ divides $f$ if and only if $q({\rm mod}\ p)\equiv 0$. (Indeed, it suffices to remember that the initial term of a SQH singularity has to be an isolated singularity.)
\end{proof}
Note that this means that if $f$ divisible by $y$, then this is the homogeneous non-convenient case and the non-degenerate deformations from~\cite{BKW} apply. The only non-convenient case left is $q=kp$, $k>1$. But the Milnor number of such a singularity is equal to the Milnor number of a convenient SQH singularity with $p,q$ the same as the original one and the deformations of Lemma~\ref{LemPKP} apply.

\begin{prop}
All Newton numbers attained by deformations of Newton diagram of a SQH singularity $f$ are equivalent to Milnor numbers attained by nondegenerate linear deformations of $f$.
\end{prop}
\begin{proof}
Let a nondegenerate singularity $f$ have the Newton diagram $D$. Let $\tilde{f}$ be the family of all nondegenerate singularities with diagram $\tilde{D}$. If $\tilde{D}$ is a deformation of $D$, then for a generic choice of coefficients $f+t\tilde{f}$ is a deformation of $f$. (Here generic can mean outside an algebraic set in the jet space $J^q$ of analytic functions in two variables that cuts transversally the Zariski closure of the family $\tilde{f}$).\end{proof}

\subsection{Auxiliary EEA sequence} As in~\cite{MMJW} we will use a sequence arising from Extended Euclid's Algorithm. This sequence lies at the heart of the combinatorial method of constructing deformations with a given Milnor number, especially if the difference compared to $\mu(f)$ is small. 

Recall that for any $a,b$ coprime and $\epsilon=\pm 1$ there exist unique positive integers $a',b'$ such that $a'<a, b'<b$ and $ab'-ba'=\epsilon$. Hence
\begin{rk}\cite[Properties 2.7 and 2.8]{MMJW}\label{rk_EEA_sequence}
For $a_0,b_0$ coprime positive integers, there exists a finite sequence of non-negative integers $(a_j,b_j)_{j=1,\dots,l}$ such that 
$$ a_{j-1}b_{j}-b_{j-1}a_j=(-1)^{l-1-j}, $$
$b_j< b_{j-1}$ for $j=1,\dots,l$  and $a_l=0$, $b_l=1$. Additionally, we may assume ${b_0\over b_1}\geq 2$ and with this condition the sequence is unique.
\end{rk}
This sequence can be easily computed from Extended Euclid's Algorithm. 

For the sequence as above denote ${\rm sign}(a_j,b_j)=(-1)^{l-1-j}$. Note that signs in the sequence alternate and ${\rm sign}(a_{l-1},b_{l-1})=1$.

\section{Lemmas}\label{section_lemmas}

\subsection{Initial jumps for $(p,q)=m<p$}

Let $a_0,b_0$ be coprime and consider the sequence $(a_j,b_j)$ satisfying the conditions of Remark~\ref{rk_EEA_sequence}. We have $a_0=Na_1+na_2$ and $b_0=Nb_1+nb_2$ for some unique positive integers $n,N$.

Let us recall that for such $n$ and $N$ we have
\begin{fact}\cite[Propositions 3.10 and 3.13]{MMJW}
\label{propSigma0DoSigmaN}
Let $a_0,b_0$ be coprime. There are deformations of a diagram $\tr(a_0,b_0)$ that give opening terms of the sequence of minimal jumps of Newton numbers
$$\underbrace{1\ ,\ \dots\ ,\ 1}_{nN}$$
\end{fact}

Moreover, retaining the notation we have

\begin{fact}\cite[Remarks 3.8 and 3.14]{MMJW}
\label{rkAllDeformationsAboveGammaN}
Let $a_0,b_0$ be coprime.

If $a_1\neq 1$, all points giving the deformations of Fact~\ref{propSigma0DoSigmaN} span the diagram
$$\Sigma_{a_0,b_0}:= {\rm sign}(a_0,b_0)\ \big(\ N \tr (a_1,b_1) + n\ \tr\left(a_2,b_2\right)\ \big).$$
If $a_1=1$, all points giving deformations of Fact~\ref{propSigma0DoSigmaN} span the diagram
$$r\tr (1,k+1)+(a_0-r)\tr(1,k),$$
where $b_0=ka_0+r$. 
\end{fact}
Note that in particular it follows that the deformation giving the smallest Newton number (in Fact~\ref{propSigma0DoSigmaN}) has exactly the diagram as in Fact~\ref{rkAllDeformationsAboveGammaN}.

Using the above

\begin{lem}\label{prop_mp_mq}
If $(p,q)=m<p\leq q$, then the opening terms of the sequence of jumps of Newton numbers for a diagram $\tr(p,q)$ are
$$m,\ \underbrace{1\ ,\ \dots\ ,\ 1}_{r(p-r)-m}$$
\end{lem}

\begin{proof}
For $p,q$ coprime the Proposition is true due to \cite[Thm 1.1]{MMJW}. Assume $p,q$ are not coprime i.e. $m>1$. We have $p=m\tilde{p}, q=m\tilde{q}$. Put $a_0=\tilde{p}$ and $b_0=\tilde{q}$ and consider the unique sequence $(a_j,b_j)_{j=1,\dots,l}$ from Remark~\ref{rk_EEA_sequence}.

Take the one-point deformation with the diagram
\begin{equation}
\label{eqFirstDiag}
{\rm sign}(a_0,b_0) \left(\  \tr(p-a_1,q-b_1)+\tr(a_1,b_1) \ \right).
\end{equation}
This deformation gives jump $m$ and no other deformation gives a greater Newton number.

Since $p-a_1$ and $q-b_1$ are coprime, use Fact~\ref{propSigma0DoSigmaN} for the segment $\tr(p-a_1,q-b_1)$ of the diagram~\eqref{eqFirstDiag}. We can do this because due to Remark~\ref{rkAllDeformationsAboveGammaN} all deformations will lie between diagram~\eqref{eqFirstDiag} and the diagram
$${\rm sign}(a_0,b_0)\ \left(\ \Sigma_{p-a_1,q-b_1}+\tr(a_1,b_1)\ \right)\ =\ {\rm sign}(a_1,b_1)\ \left(\  mN_1\tr(a_1,b_1)+mn_1\tr(a_2,b_2)\ \right),$$
which is convex. Here $n_1,N_1$ are natural numbers such that $b_0=N_1b_1+n_1b_2$. Hence the deformations are indeed deformations of the initial diagram $\tr(p,q)$. Moreover, since the end-points remain fixed, the differences in Newton numbers are preserved.

Inductively, apply Fact~\ref{propSigma0DoSigmaN} to segments $\tr(mN_ia_i+a_{i+1}, mN_ib_i+b_{i+1})$ in the diagrams
$$D_i= -{\rm sign}(a_i,b_i)\ \big(\ \tr(mN_ia_i+a_{i+1}, mN_ib_i+b_{i+1})+(mn_i-1)\tr(a_{i+1},b_{i+1})\ \big),$$
where $N_ib_i+n_ib_{i+1}=b_0$. From Fact~\ref{rkAllDeformationsAboveGammaN} we get that deformations of each $D_i$ lie on or above the diagram 
$$E_i=-{\rm sign}(a_i,b_i)\ \big(\ \Sigma_{mN_ia_i+a_{i+1}, mN_ib_i+b_{i+1}}+(mn_i-1)\tr(a_{i+1},b_{i+1})\ \big),$$
which gives the smallest Newton number. In fact 
$$E_i={\rm sign}(a_i,b_i)\ \big(\ mN_{i+1}\tr(a_{i+1},b_{i+1})+mn_{i+1}\tr(a_{i+2},b_{i+2})\ \big)$$
and $D_{i+1}\geq E_i$. This gives the inequality $\nu(D_{i+1})\geq \nu(E_i)$ between Newton numbers. Therefore between every inductive step there is no gap greater than one.

Note that the last deformation due to Remark~\ref{rk_EEA_sequence} and the second part of Fact~\ref{rkAllDeformationsAboveGammaN} will have the diagram of the form
$$L=r\tr(1,k+1)+(p-r)\tr(1,k).$$
The above is easy to check since end-points remain fixed. Hence in the sequence of Newton numbers of deformations there is a sequence of consecutive numbers from $\nu(\tr(p,q))-m$ to $\nu(L)=\nu(\tr(p,q))-r(p-r)$. This ends the proof.
\end{proof}

It is convenient to underline the fact below.
\begin{rk}
The Newton diagram of the sum of supports of deformations used in above Proposition~\ref{prop_mp_mq} is equal
$$ r\tr(1,k+1)+(p-r)\tr(1,k) $$
where $q=kp+r$. In particular, this is the diagram of the deformation giving the smallest Newton number on the list.
\end{rk}

\subsection{Initial jumps for $mp,mq$ even longer $r(p-1)$}

\begin{lem}\label{lemSkok2pParzyste}
If $q\equiv p-1 ({\rm mod}\ p)$, we have $(p,q-1)\leq 2$ and
$$(p,q-1)>1\iff p {\rm \ is\ even.} $$
\end{lem}
\begin{proof}
Let $p=dp_0$, $q-1=dq_0$. Obviously $d<p$. We have $q-1=(k+1)p-2$, hence $dq_0=d(k+1)p_0-2$. Therefore, either $d=2$ and $q_0-(k+1)p_0=1$ or $d=1$ and $q_0-(k+1)p_0=2$. The first case can occur if and only if $p$ is even, the second if and only if $p$ is odd (otherwise both $p$ and $q-1$ are even and $d=2$). 
\end{proof}

\begin{lem}\label{lemDluzszePoczatkowe}
Consider $4<p\leq q$. Denote $m=(p,q), q=kp+r$, $0<r<p$. For the Newton diagram with end-points $(0,q), (p,0)$ all Newton numbers between $\nu(\tr(p,q))$ and $\nu(\tr(p,kp))$ are attained except:

numbers between $\nu(\tr(p,q))$ and $\nu(\tr(p,q))-m$ when $p$ is odd or $q\nequiv p-1 ({\rm mod}\ p)$ 

or

the number $\nu(\tr(p,q))-p$ when $p$ is even AND $q\equiv p-1 ({\rm mod}\ p)$.
\end{lem}
\begin{proof}
If $r=0$, the theorem is trivial. If $r=1$, then $p,q$ are coprime and use Lemma~\ref{prop_mp_mq} to get the claim.

Let us assume $q\nequiv \pm 1$. We will consider deformations of diagrams
$$\tr(p,q),\  \tr(p,q-1),\ \dots\ ,\  \tr(p,q-(r-1))$$
with both end-points on the axes.

For any $l=0,\dots,r-1$ we have 
\begin{equation}\label{eqPomocNewton}
\nu(\tr(p,q-l))=\nu(\tr(p,q)) - l(p-1)
\end{equation}
and $q-l\equiv r-l>0$, hence $p$ does not divide $q-l$. Denote $m_l=(p,q-l)$. From Lemma~\ref{prop_mp_mq} all numbers from $\nu(\tr(p,q-l))$ to $\nu(\tr(p,q-l))  - (r-l)(p-(r-l))$ are attained except for those between $\nu(\tr(p,q-l))$ and $\nu(\tr(p,q-l))-m_l$. 

We will make sure that these sequences give consecutively the sequence of jumps equal $1$, i.e. that the missing Newton numbers in each step $l$ are already covered by Newton numbers in the preceding step.

First, we show that $\nu(p,q-l)-m_l\geq \nu(p,q-(l+1))$. Indeed, since $m_l<p$, we get $-l(p-1)\geq -(l+1)(p-1)+m_l$ and from equality~\eqref{eqPomocNewton} we get the claim.

Now it suffices to show that the first attained number after the gap in step $l+1$ is bigger than the last number attained in step $l$ i.e.
\begin{equation}
\label{eqPomocNewton2}
\nu(\tr(p,q-(l+1)))-m_{l+1} \geq \nu(\tr(p,q-l)) - (r-l)(p-(r-l))
\end{equation}
for $l=0,\dots, r-2$. Indeed, note that $2\leq r-l \leq p-2$ and therefore 
$$\min_{l=0,\dots,r-2}(r-l)(p-(r-l))\geq 2(p-2).$$ 
Since $m_{l+1}<p$ and it divides $p$, we get $p-3\geq m_{l+1}$ for $p\geq 5$. Hence
$$(r-l)(p-(r-l))-(p-1)\geq 2(p-2)-(p-1)=p-3\geq m_{l+1}.$$
Combine with the fact $\nu(\tr(p,q-(l+1)))=\nu(\tr(p,q-l))-(p-1)$ and we get inequality~\eqref{eqPomocNewton2}.

Therefore we get that all numbers between $\nu(\tr(p,q))-m$ and $\nu(\tr(p,q-(r-1)))-(p-1)=\nu(\tr(p,q))-r(p-1)$ are attained which gives the claim of the lemma.

Now assume that $q\equiv p-1$. In particular, $p,q$ are coprime and Lemma~\ref{prop_mp_mq} shows that all numbers from $\nu(\tr(p,q))$ to $\nu(\tr(p,q))-(p-1)$ are attained. For $p,q-1$ we have $q-1\equiv p-2\neq p-1$, hence we can apply the reasoning above. Note that $\nu(\tr(p,q-1))=\nu(\tr(p,q))-(p-1)$ and $(p,q-1)\leq 2$ depending on whether $p$ is even or odd due to Lemma~\ref{lemSkok2pParzyste}, hence all numbers from $\nu(\tr(p,q))$ to $\nu(\tr(p,q))-(p-1)(p-2)-(p-1)$ are attained except for $\nu(\tr(p,q))-p$ if $p$ is even. This ends the proof.
\end{proof}

\section{Main results combinatorially}\label{section_Main_combinatorial}

\begin{thm}
\label{LemPKP}
Consider $p$ and $q=kp$. For any Newton diagram contained in the diagram $\tr(p,kp)$ with end-points $(0,kp), (p,0)$ all numbers are attained except\\
 numbers between $\nu(\tr(p,kp))$ and $\nu(\tr(p,kp))-(p-1)$\\
 and\\
the number $\nu(\tr(p,q))-(2p-1)$ when $p$ is even.
\end{thm}
\begin{proof} Assume $p>4$, $k>1$ and that the diagram is equal $\tr(p,kp)$ with end-points on both axes.

For any $\kappa\in\Na$ first jump for $\tr(p,\kappa p)$ is $p-1$ attained by deformation $\tr(p,\kappa p-1)$. Since $\kappa p-1\equiv p-1 ({\rm mod}\ p)$, we use Lemma~\ref{lemDluzszePoczatkowe} and get all numbers from $\nu(\tr(p,\kappa p-1))=\nu(\tr(p,\kappa p))-(p-1)$ to $\nu(\tr(p,\kappa p-1))-(p-1)^2=\nu(\tr(p,\kappa p))-p(p-1)$ except $\nu(\tr(p,\kappa p-1))-p=\nu(\tr(p,\kappa p))-(2p-1)$ when $p$ even. Note that $\nu(\tr(p,\kappa p))-p(p-1)=\nu(\tr(p,(\kappa -1)p))$ for $\kappa \geq 2$.

Consider diagrams
\begin{equation}\label{eq_deform_p-1}
\tr(2,p+2\kappa )+\tr(i-2,(i-2)\kappa )+\tr(p-i,\kappa (p-i)-1), 
\end{equation}
for $i=2,\ldots,p-1$ with end-points $(p,0)$ and $(0,(\kappa +1)p-1)$. For $\kappa <k$ they are deformations of $\tr(p,kp)$. 

For $\kappa \geq 2$ the above deformations~\eqref{eq_deform_p-1} give all numbers between $\nu(\tr(p,\kappa p))$ and $\nu(\tr(p,\kappa p))-(p-1)$. If $\kappa =1$ 
then above deformations give all numbers between $\nu(\tr(p,p))$ and $\nu(\tr(p,p))-(p-1)+1=\nu(\tr(p,p-1))+2$.
The number $\nu(\tr(p,p-1))+1$ is attained by the deformation 
$\tr(2,p+6)+\tr(p-2,p-4)$ of $\tr(p,kp)$.

Moreover, one can check that the deformation $\nu(\tr(2,2\kappa+3))+\tr(p-2,(p-2)\kappa -2 )$ gives the number $\nu(p,\kappa p)-(2p-1)$ for $\kappa <k$. 

Hence by induction for any $k\geq 2$ we get that all numbers from $\nu(\tr(p,kp))$ to $1$ are attained except for:\\
numbers between $\nu(\tr(p,kp))$ and $\nu(\tr(p,kp))-(p-1)$ \\
AND\\
$\nu(\tr(p,kp))-(2p-1)$ if $p$ is an even number.

For $p=1$ the germ $f$ is smooth. Now note that for $p=2,3,4$ the claim holds, see Lemma~\ref{rk_dla_234} below. To conclude the proof, note that the Newton number of a Newton diagram contained in $\tr(p,kp)$ with end-points on both axes has the same Newton number as the bigger diagram. Moreover, all deformations in the considerations above are deformations also of such a smaller diagram. Moreover, for $k=1$ the claim holds also, see~\cite[Thm BKW nzdeg]{BKW} and Remark~\ref{dla $p<5$} below. This ends the proof.\end{proof}

\begin{rk}\label{dla $p<5$}
Consider $\tr(p,p)$, $4\geq p\geq 2$. All numbers between $\nu(p,p)-(p-1)$ and $1$ are attained. Indeed, the first non-degenerate jump is equal to $p-1$. One can check by hand that all numbers from $\nu(\tr(p,p-1))=\nu(\tr(p,p))-(p-1)$ to $1$ are attained.
\end{rk}

The following Lemma can be proven to great extent by methods of proof of Theorem~\ref{LemPKP}. Nevertheless, we thought it might be instructive to provide explicit deformations in the case of small $p$.

\begin{lem}\label{rk_dla_234}
Consider a diagram $\tr(p,kp)$ for $p=2,3,4$ and $k\geq 2$. The claim of Theorem~\ref{LemPKP} holds.
\end{lem}
\begin{proof}
We will show explicitly the deformations needed depending on $p$. Take any positive integer $\kappa \leq k$.

Consider $p=2$. For $\tr(2,2k)$ by deformations $\tr(2,2k-i)$, $i=1,\ldots,2k-2$ we get all numbers from $\nu(\tr(2,2k))$ to $1$.

Consider $p=3$. The first jump for $\tr(3,3k)$ is $2$ attained by deformation $\tr(3,3k-1)$ and
\begin{itemize}

\item deformations $i\tr(1,\kappa )+\tr(3-i,\kappa (3-i) -1 )$, $i=0,1,2$ give numbers from $\nu(\tr(3,3\kappa -1))$ to $\nu(\tr(3,3\kappa -1))-2$ for $\kappa \geq 2$

\item the deformation $\tr(2,2\kappa -1)+\tr(1,\kappa -1)$ gives the number $\nu(\tr(3,3\kappa -1))-3$ for $\kappa \geq 2$

\item the deformation $\tr(3,3\kappa -3)$ gives the number $\nu(\tr(3,3\kappa -1))-4$ for for $\kappa \geq 2$

\item deformations  $\tr(2,2\kappa +1)+\tr(1,\kappa -2)$ for $\kappa \geq 2$ and $\tr(2,4)$ for $\kappa =2$ give the number $\nu(\tr(3,3\kappa -1))-5$
\end{itemize}
 Hence by Remark \ref{dla $p<5$} we get that for $p=3$ all numbers from $\nu(\tr(p,kp))$ to $1$ are attained except for $\nu(\tr(3,3k))-1$.

Consider $p=4$. The first jump for $\tr(4,4k)$ is $3$ attained by deformation $\tr(4,4k-1)$ and

\begin{itemize}
\item deformations $i\tr(1,\kappa )+\tr(4-i,(4-i)\kappa -1 )$  for $i=0,\ldots,3$ give numbers from $\nu(\tr(4,4\kappa -1))$ to $\nu(\tr(4,4\kappa -1))-3$ for $\kappa \geq 2$.

\item deformations $i\tr(1,\kappa )+\tr(3-i,(3-i)\kappa -1 )+\tr(1,\kappa -1)$ for $i=0,1,2$
give numbers $\nu(\tr(4,4\kappa -1))-5$, $\nu(\tr(4,4\kappa -1))-6$, $\nu(\tr(4,4\kappa -1))-7$ for $\kappa \geq 2$.

\item deformations $i\tr(1,\kappa )+\tr(2-i,(2-i)\kappa -1 )+\tr(2,2\kappa -2)$ for $i=0,1$ give numbers $\nu(\tr(4,4\kappa -1))-8$, $\nu(\tr(4,4\kappa -1))-9$ for $\kappa \geq 2$.
Note that $\nu(\tr(4,4(\kappa -1)))=\nu(\tr(4,4\kappa -1))-9$.

\item deformations $\tr(2,2\kappa +3-i)+\tr(2,2\kappa -3)$ for $i=1,2$ give numbers $\nu(\tr(4,4(\kappa -1)))-1$ and $\nu(\tr(4,4(\kappa -1)))-2$ for $\kappa \geq 2$.

\item deformations $\tr(2,2\kappa +1)+\tr(2,2\kappa -4)$ for $\kappa >2$ and $\tr(2,3)$ for $\kappa =2$ give the number  $\nu(\tr(4,4(\kappa -1)))-7$. \\
(Note that the deformation $\tr(3,8)$ gives the number $\nu(\tr(4,4k))-7$ for $k=2$ but for $k>2$ there is no deformation which gives this number.) 
\end{itemize}
Hence by Remark \ref{dla $p<5$} we get that for $p=4$ all numbers from $\nu(\tr(p,kp))$ to $1$ are attained except for  $\nu(\tr(4,4k))-1$, $\nu(\tr(4,4k))-2$ and $\nu(\tr(4,4k))-7$ (the last gap disappears when $k=2$).
\end{proof}

\begin{thm}
Consider $4< p<q$. Denote $q= kp+r$, $0<r<p$ and $m=(p,q)$. All numbers from $\nu(\tr(p,q))$ to $1$ are attained as Newton numbers except for at most:\\
numbers between $\nu(\tr(p,q))$ and $\nu(\tr(p,q))-m$\\
and\\
numbers between $\nu(p,kp)$ and $\nu(\tr(p,kp))-(p-1)$\\
and\\
the number $\nu(p,q)-p$ if $p$ even and $r=p-1$\\
and\\
the number $\nu(\tr(p,kp))-(2p-1)$ if $p$ is even.
\end{thm}
\begin{proof}
Combine Lemma~\ref{lemDluzszePoczatkowe} and Theorem~\ref{LemPKP} to get the claim.
\end{proof}

\section*{Aknowledgements} Authors were supported by grant NCN 2013/09/D/ST1/03701.

\bibliographystyle{alpha}
\bibliography{bibliografia}

\begin{thebibliography}{BKW14}

\bibitem[Arn04]{ArnoldProblems}
Vladimir~I. Arnold.
\newblock {\em Arnold's problems}.
\newblock Springer-Verlag, Berlin; PHASIS, Moscow, 2004.
\newblock Translated and revised edition of the 2000 Russian original, With a
  preface by V. Philippov, A. Yakivchik and M. Peters.

\bibitem[BK14]{BK}
Szymon Brzostowski and Tadeusz Krasi{\'n}ski.
\newblock The jump of the {M}ilnor number in the {$X_9$} singularity class.
\newblock {\em Cent. Eur. J. Math.}, 12(3):429--435, 2014.

\bibitem[BKW14]{BKW}
S.~{Brzostowski}, T.~{Krasinski}, and J.~{Walewska}.
\newblock {Milnor numbers in deformations of homogeneous singularities}.
\newblock {\em ArXiv}, 2014.

\bibitem[Bod07]{Bo}
Arnaud Bodin.
\newblock Jump of {M}ilnor numbers.
\newblock {\em Bull. Braz. Math. Soc. (N.S.)}, 38(3):389--396, 2007.

\bibitem[GLS07]{GreuelShustinL}
G.-M. Greuel, C.~Lossen, and E.~Shustin.
\newblock {\em Introduction to singularities and deformations}.
\newblock Springer Monographs in Mathematics. Springer, Berlin, 2007.

\bibitem[GZ93]{GuseinZade}
S.~M. Guse{\u\i}n-Zade.
\newblock On singularities that admit splitting off {$A_1$}.
\newblock {\em Funktsional. Anal. i Prilozhen.}, 27(1):68--71, 1993.

\bibitem[Kou76]{Kush}
A.~G. Kouchnirenko.
\newblock Poly\`edres de {N}ewton et nombres de {M}ilnor.
\newblock {\em Invent. Math.}, 32(1):1--31, 1976.

\bibitem[MW14]{MMJW}
M.~{Michalska} and J.~{Walewska}.
\newblock {Milnor numbers of deformations of semi-quasi-homogeneous plane curve
  singularities}.
\newblock {\em ArXiv e-prints}, 2014.

\bibitem[P{\l}o14]{Pl14}
Arkadiusz P{\l}oski.
\newblock A bound for the {M}ilnor number of plane curve singularities.
\newblock {\em Cent. Eur. J. Math.}, 12(5):688--693, 2014.

\end{thebibliography}

\end{document}